\documentclass[a4paper,12pt]{article}
\usepackage{amsmath, amsthm, amsfonts, amssymb, graphicx, color, cancel, ulem}
\usepackage{esint}
\numberwithin{equation}{section}
     \newtheorem{thm}{Theorem}[section]
     \newtheorem{cor}[thm]{Corollary}
     
     \newtheorem{lem}[thm]{Lemma}
\theoremstyle{definition}
      \newtheorem{defn}{Definition}[section]
     \newtheorem{exmp}{Example}[section]
\theoremstyle{remark}
     \newtheorem{rem}{Remark}[section]

\newcommand{\R}{\mathbb{R}}

\newcommand{\cS}{\mathcal{S}}

\newcommand{\cY}{\mathcal{Y}}

\newcommand{\supp}{\operatorname{supp}}

\newcommand{\comp}{\mathrm{comp}}
\newcommand{\weak}{\mathrm{weak}}

\newcommand{\ls}{\lesssim}
\newcommand{\gs}{\gtrsim}

\newcommand{\vp}{\varphi}

\newcommand{\cPhi}{\widetilde\Phi}

\newcommand{\cGdec}{\mathcal{G}^{\rm dec}}

\newcommand{\Cic}{C^{\infty}_{\comp}}
\newcommand{\Li}{L^{\infty}}

\newcommand{\LP}{L^{\Phi}}

\newcommand{\LcP}{L^{\widetilde\Phi}}

\newcommand{\wL}{\mathrm{w}\hskip-0.6pt{L}}
\newcommand{\wLP}{\mathrm{w}\hskip-0.6pt{L}^{\Phi}}

\newcommand{\LPp}{L^{(\Phi,\vp)}}

\newcommand{\wLPp}{\mathrm{w}\hskip-0.6pt{L}^{(\Phi,\vp)}}

\newcommand{\LPsps}{L^{(\Psi,\psi)}}

\newcommand{\wLPsps}{\mathrm{w}\hskip-0.6pt{L}^{(\Psi,\psi)}}

\newcommand{\biP}{{\it{\overline{\boldsymbol \varPhi}}}}
\newcommand{\iPy}{{{\boldsymbol \varPhi}_Y}}
\newcommand{\bcY}{{\overline{\cY}}}

\newcommand{\biPy}{{\overline{\boldsymbol \varPhi}_Y}}

\newcommand{\dtwo}{\Delta_2}

\newcommand{\bdtwo}{\overline\Delta_2}
\newcommand{\bntwo}{\overline\nabla_2}

\newcommand{\dlim}{\displaystyle\lim}

\newcommand{\msckw}{%
\footnotetext{\hspace{-0.35cm} 2020 {\it Mathematics Subject Classification}. 
46E30, 42B35, 42B25, 42B20.
\endgraf{\it Key words and phrases.} 
Orlicz-Morrey space, modular inequality, maximal function, singular integral. \par
Ryota Kawasumi,
Minohara 1-6-3 (B-2), Misawa, Aomori 033-0033, Japan, \\
rykawasumi@gmail.com  \par
Eiichi Nakai,
Department of Mathematics,
Ibaraki University,
Mito, Ibaraki 310-8512, Japan,
eiichi.nakai.math@vc.ibaraki.ac.jp
}
}

\begin{document}

\title{%
Weighted boundedness of the Hardy-Littlewood maximal and
Calder\'on-Zygmund operators on Orlicz-Morrey and weak Orlicz-Morrey spaces
\msckw
}
\author{Ryota Kawasumi and Eiichi Nakai}
\date{}

\maketitle

\begin{abstract}
For the Hardy-Littlewood maximal and Calder\'on-Zygmund operators,
the weighted boundedness on the Lebesgue spaces are well known.
We extend these to the Orlicz-Morrey spaces.
Moreover, we prove the weighted boundedness on the weak Orlicz-Morrey spaces.
To do this we show the weak-weak modular inequality.
The Orlicz-Morrey space and its weak version contain 
weighted Orlicz, Morrey and Lebesgue spaces and their weak versions as special cases.
Then we also get the boundedness for these function spaces as corollaries.
\end{abstract}

\section{Introduction}\label{sec:intro}

Let $L^p(\R^n,w)$ and $\wL^p(\R^n,w)$ be the weighted Lebesgue space 
and its weak version
on the $n$-dimensional Euclidean space $\R^n$,
respectively.
Then it is well known that
the Hardy-Littlewood maximal operator $M$ is bounded
from $L^1(\R^n,w)$ to $\wL^1(\R^n,w)$ if $w\in A_1$,
and, 
from $L^p(\R^n,w)$ to itself if $w\in A_p$, $p\in(1,\infty]$,
where $A_p$ is the Muckenhoupt class, see \cite{Muckenhoupt1972}.
The Calder\'on-Zygmund operators have the same boundedness except the case $p=\infty$.
It is also known that $M$ is bounded
from $\wL^p(\R^n,w)$ to itself if $w\in A_p$, $p\in(1,\infty]$.
This boundedness can be obtained by using the property of $A_p$-weights
and  the Marcinkiewicz interpolation theorem
for the operators of restricted weak type,
see \cite[Theorem 1.4.19 (page 61)]{Grafakos2014GTM249} for example.
See also \cite{Komori2012} for its simple proofs.

In this paper we extend these boundedness to 
the weighted Orlicz-Morrey space $\LPp(\R^n,w)$ 
and its weak version $\wLPp(\R^n,w)$,
where $\Phi$ is a Young function and $\vp$ is a variable growth function. 
Namely, we prove the following boundedness:
\begin{align}
 \|Tf\|_{\wLPp(\R^n,w)}&\le C\|f\|_{\LPp(\R^n,w)}, \label{w} \\
 \|Tf\|_{\LPp(\R^n,w)}&\le C\|f\|_{\LPp(\R^n,w)}, \label{s} \\
 \|Tf\|_{\wLPp(\R^n,w)}&\le C\|f\|_{\wLPp(\R^n,w)}, \label{ww}
\end{align}
where $T$ is 
the Hardy-Littlewood maximal operator or a Calder\'on-Zygmund operator.
The function spaces $\LPp(\R^n,w)$ and $\wLPp(\R^n,w)$ contain 
weighted Orlicz, Morrey and Lebesgue spaces and their weak versions as special cases.
Then we also get the boundedness for these function spaces as corollaries.

For a measurable set $G\subset\R^n$, we denote its Lebesgue measure 
and characteristic function by $|G|$ and $\chi_G$, respectively.
A weight is a locally integrable function on $\R^n$ 
which takes values in $(0,\infty)$ almost everywhere. 
For a weight $w$ and a measurable set $G$, 
we define $w(G) = \int_G w(x)\,dx$.
For $p\in(0,\infty]$, the weighted Lebesgue space and 
its weak version with respect to the measure $w(x)\,dx$ 
are denoted by $L^p(\R^n,w)$ and $\wL^p(\R^n,w)$, respectively.

For a function $\vp:\R^n\times(0,\infty)\to(0,\infty)$
and a ball $B=B(a,r)$, 
we denote $\vp(a,r)$ by $\vp(B)$. 
For a weight $w$, a measurable set $G$ and a function $f$, 
let
\begin{equation*}
 w(G,f,t)
 =
 w(\{x\in\ G:|f(x)|>t\}),
 \quad t\in[0,\infty).
\end{equation*}
In the case $G=\R^n$, we briefly denote it by $w(f,t)$.

\begin{defn}[Orlicz-Morrey space and weak Orlicz-Morrey space]\label{defn:OM}
For a Young function $\Phi:[0,\infty]\to[0,\infty]$, 
a function $\vp:\R^n\times(0,\infty)\to(0,\infty)$, a weight $w$
and a ball $B$, 
let
\begin{align*} 
 \|f\|_{\Phi,\vp,w,B}
 &= 
 \inf\left\{ \lambda>0: 
  \frac{1}{\vp(B)w(B)}
  \int_B \!\Phi\!\left(\frac{|f(x)|}{\lambda}\right)\! w(x)\,dx \le 1
           \right\},
\\
 \|f\|_{\Phi,\vp,w,B,\weak }
 &= 
 \inf\left\{ \lambda>0: 
  \frac{1}{\vp(B)w(B)}
 \sup_{t\in (0,\infty)} \Phi(t)\, w\!\left( B,\frac{f}{\lambda}, t\right)
 \le 1
           \right\}.
\end{align*}
Let $\LPp(\R^n,w)$ and $\wLPp(\R^n,w)$ be the sets of all functions $f$ 
such that the following functionals are finite, respectively: 
\begin{align*} 
 \|f\|_{\LPp(\R^n,w)} 
 &=
 \sup_{B}  \|f\|_{\Phi,\vp,w,B},
\\
 \|f\|_{\wLPp(\R^n,w)} 
 &=
 \sup_{B}  \|f\|_{\Phi,\vp,w,B,\weak},
\end{align*}
where the suprema are taken over all balls $B$ in $\R^n$.
(For the definition of the Young function, see the next section.)
\end{defn}

Then $\|f\|_{\LPp(\R^n,w)}$ is a norm and thereby $\LPp(\R^n,w)$ is a Banach space,
and $\|f\|_{\wLPp(\R^n,w)}$ is a quasi norm and thereby $\wLPp(\R^n,w)$ is a quasi Banach space.
The Orlicz-Morrey space $\LPp(\R^n)$ was first studied in \cite{Nakai2004KIT}.
The spaces $\LPp(\R^n)$ and $\wLPp(\R^n)$ 
were investigated 
in \cite{Ho2013,Nakai2008Studia,Nakai2008KIT,Sawano2019}, etc.
For other kinds of Orlicz-Morrey spaces, see 
\cite{Deringoz-Guliyev-Nakai-Sawano-Shi2019Posi,Deringoz-Guliyev-Samko2014,Gala-Sawano-Tanaka2015,Guliyev-Hasanov-Sawano-Noi2016,Sawano-Sugano-Tanaka2012}, etc.
See also \cite{Ho2019,Ho2020} for Morrey-Banach spaces. 

The function spaces $\LPp(\R^n,w)$ and $\wLPp(\R^n,w)$ contain several function spaces 
as special cases.
If $\vp(B)=1/w(B)$, 
then $\LPp(\R^n,w)$ and $\wLPp(\R^n,w)$
coincide with the weighted Orlicz space $\LP(\R^n,w)$ and its weak version $\wLP(\R^n,w)$,
respectively.
If $\Phi(t)=t^p$, $1\le p<\infty$, 
then $\LPp(\R^n,w)$ and $\wLPp(\R^n,w)$ are denoted by 
$L^{(p,\vp)}(\R^n,w)$ and $\wL^{(p,\vp)}(\R^n,w)$, respectively,
which are the generalized weighted Morrey space and its weak version.
If $\vp(B)=w(B)^{\kappa-1}$, $0<\kappa<1$, 
then $L^{(p,\vp)}(\R^n,w)$ and $\wL^{(p,\vp)}(\R^n,w)$ 
are denoted by $L^{p,\kappa}(\R^n,w)$ and $\wL^{p,\kappa}(\R^n,w)$, respectively,
which were introduced by Komori and Shirai~\cite{Komori-Shirai2009}.
If $\Phi(t)=t^p$, $1\le p<\infty$, and $\vp(B)=1/w(B)$,
then $\LPp(\R^n,w)=L^p(\R^n,w)$ 
and $\wLPp(\R^n,w)=\wL^p(\R^n,w)$.
Therefore, 
by \eqref{w}, \eqref{s} and \eqref{ww},
we also have the norm inequalities for these function spaces as corollaries.

Let $A_p$ be the Muckenhoupt class of weights, see Definition~\ref{defn:Ap}.
Let $p\in[1,\infty)$ and $w$ is a weight.
Muchenhoupt~\cite{Muckenhoupt1972} proved that
the Hardy-Littlewood maximal operator $M$ is bounded 
from $L^p(\R^n,w)$ to $\wL^p(\R^n,w)$ if and only if $w\in A_p$.
He also proved that, for $p\in(1,\infty)$, 
$M$ is bounded from $L^p(\R^n,w)$ to itself if and only if $w\in A_p$.
For the boundedness of the Hiibert transform, the same conclusions hold,
see \cite{Hunt-Muckenhoupt-Wheeden1973}.

Let $p\in(1,\infty)$ and $w$ is a weight. 
It is also known that 
$M$ is bounded from $\wL^p(\R^n,w)$ to itself if and only if $w\in A_p$.
We learned from \cite{Komori2012}
two kinds of simple proofs of this boundedness
by Grafakos and by Yabuta.
By our results, we see that 
any Calder\'on-Zygmund operator
is bounded from $\wL^p(\R^n,w)$ to itself if $w\in A_p$.
In particular,
the Reisz transforms are bounded from $\wL^p(\R,w)$ to itself 
if and only if $w\in A_p$, see Corollary~\ref{cor:modular}.

In the next section 
we state on the functions $\Phi$, $\vp$ and $w$ by which we define 
$\LPp(\R^n,w)$ and $\wLPp(\R^n,w)$.
Then we state main results in Section~\ref{sec:main}.
We recall the properties of Young functions
and show a lemma in Section~\ref{sec:prop}.
To prove the norm inequalities \eqref{w}, \eqref{s} and \eqref{ww}
we need the modular inequalities
\begin{align*}
 \sup_{t\in (0,\infty)}\Phi(t)w(Tf,t)
 &\le
 C \int_{\R^n}\Phi(c|f(x)|)w(x)\,dx,
\\
 \int_{\R^n}\Phi(|Tf(x)|)w(x)\,dx
 &\le
 C \int_{\R^n}\Phi(c|f(x)|)w(x)\,dx,
\\
 \sup_{t\in (0,\infty)}\Phi(t)w(Tf,t)
 &\le
 C\sup_{t\in (0,\infty)}\Phi(t)w(f,t),
\end{align*}
respectively,
in which the first and the second are known.
We prove the third in Section~\ref{sec:modular}.
Then, using the results in Sections~\ref{sec:prop} and \ref{sec:modular},
we prove the main results in Section~\ref{sec:proof}.
In the above, each modular inequality means that 
it holds for any function $f$ 
such that the left-hand side is finite, and that the constant $C$ is independent of $f$.
We will make similar abbreviated statements involving other modular and (quasi-)norm inequalities; 
they will be always interpreted in the same way.

At the end of this section, we make some conventions. 
Throughout this paper, we always use $C$ to denote a positive constant 
that is independent of the main parameters involved 
but whose value may differ from line to line.
Constants with subscripts, such as $C_p$, is dependent on the subscripts.
If $f\le Cg$, we then write $f\ls g$ or $g\gs f$; 
and if $f \ls g\ls f$, we then write $f\sim g$.

\section{On the functions $\Phi$, $\vp$ and $w$}\label{sec:Young}

In this section 
we state on the functions $\Phi$, $\vp$ and $w$ by which we define 
$\LPp(\R^n,w)$ and $\wLPp(\R^n,w)$.
We first recall the Young function and its generalization.

For an increasing (i.e. nondecreasing) function 
$\Phi:[0,\infty]\to[0,\infty]$,
let
\begin{equation*}\label{aP bP} 
 a(\Phi)=\sup\{t\ge0:\Phi(t)=0\}, \quad 
 b(\Phi)=\inf\{t\ge0:\Phi(t)=\infty\},
\end{equation*} 
with convention $\sup\emptyset=0$ and $\inf\emptyset=\infty$.
Then $0\le a(\Phi)\le b(\Phi)\le\infty$.

Let $\biP$ be the set of all increasing functions
$\Phi:[0,\infty]\to[0,\infty]$
such that
\begin{enumerate}
\item\label{ab} \ 
$\displaystyle
 0\le a(\Phi)<\infty, \quad 0<b(\Phi)\le\infty, 
$
\item \label{lim_0} \ 
$\displaystyle
 \lim_{t\to+0}\Phi(t)=\Phi(0)=0,
$
\item \label{left cont} \ 
$\displaystyle
 \text{$\Phi$ is left continuous on $[0,b(\Phi))$}, 
$
\item \label{left cont infty} \ 
$\displaystyle
 \text{if $b(\Phi)=\infty$, then } 
 \lim_{t\to\infty}\Phi(t)=\Phi(\infty)=\infty, 
$
\item \label{left cont b} \ 
$\displaystyle
 \text{if $b(\Phi)<\infty$, then } 
 \lim_{t\to b(\Phi)-0}\Phi(t)=\Phi(b(\Phi)) \ (\le\infty). 
$
\end{enumerate}

In what follows,
if an increasing and left continuous function $\Phi:[0,\infty)\to[0,\infty)$ satisfies
\ref{lim_0} and $\dlim_{t\to\infty}\Phi(t)=\infty$,
then we always regard that $\Phi(\infty)=\infty$ and that $\Phi\in\biP$.

For $\Phi, \Psi\in\biP$, 
we write $\Phi\approx\Psi$
if there exists a positive constant $C$ such that
\begin{equation*} 
     \Phi(C^{-1}t)\le\Psi(t)\le\Phi(Ct)
     \quad\text{for all}\ t\in[0,\infty].
\end{equation*} 

Now we recall the definition of the Young function and give its generalization.

\begin{defn}\label{defn:Young}
\begin{enumerate}
\item 
A function $\Phi\in\biP$ is called a Young function 
(or sometimes also called an Orlicz function) 
if $\Phi$ is convex on $[0,b(\Phi))$.
Let $\iPy$ be the set of all Young functions,
and let $\biPy$ be the set of all $\Phi\in\biP$ such that
$\Phi\approx\Psi$ for some $\Psi\in\iPy$.
(Each $\Phi\in\biPy$ is also called a quasi-convex function, see~\cite{Kokilashvili-Krbec1991}).
\item
Let $\cY$ be the set of all Young functions such that $a(\Phi)=0$ and $b(\Phi)=\infty$,
and let $\bcY$ be the set of all $\Phi\in\biP$ such that
$\Phi\approx\Psi$ for some $\Psi\in\cY$.
\end{enumerate}
\end{defn}

By the convexity, 
any Young function $\Phi$ is continuous on $[0,b(\Phi))$ 
and strictly increasing on $[a(\Phi),b(\Phi)]$.
Hence $\Phi$ is bijective from $[a(\Phi),b(\Phi)]$ to $[0,\Phi(b(\Phi))]$.
If $\Phi\in\cY$, then $\Phi$ is continuous and bijective from $[0,\infty]$ to itself.

\begin{defn}\label{defn:D2 n2}
\begin{enumerate}
\item 
A function $\Phi\in\biP$ is said to satisfy the $\Delta_2$-condition,
denoted by $\Phi\in\bdtwo$, 
if there exists a constant $C>0$ such that
\begin{equation*} 
 \Phi(2t)\le C\Phi(t) 
 \quad\text{for all } t>0.
\end{equation*}
\item
A function $\Phi\in\biP$ is said to satisfy the $\nabla_2$-condition,
denoted by $\Phi\in\bntwo$, 
if there exists a constant $k>1$ such that
\begin{equation*} 
 \Phi(t)\le\frac1{2k}\Phi(kt) 
 \quad\text{for all } t>0.
\end{equation*}
\item
Let $\Delta_2=\iPy\cap\bdtwo$ and $\nabla_2=\iPy\cap\bntwo$.
\end{enumerate}
\end{defn}

For $\Phi\in\biP$, 
we recall the dilation indices which are also called 
the Orlicz-Matuszewska-Maligranda indices:

\begin{defn}\label{defn:index}
For $\Phi\in\biP$ with $a(\Phi)=0$ and $b(\Phi)=\infty$,
let
\begin{equation*}
 h_{\Phi}(\lambda)=\sup_{t\in(0,\infty)}\frac{\Phi(\lambda t)}{\Phi(t)},
 \quad \lambda\in(0,\infty),
\end{equation*}
and define the lower and upper indices of $\Phi$ by
\begin{align*}
 i(\Phi)
 &=
 \lim_{\lambda\to+0}\frac{\log h_{\Phi}(\lambda)}{\log\lambda}
 =
 \sup_{\lambda\in(0,1)}\frac{\log h_{\Phi}(\lambda)}{\log\lambda},
\\
 I(\Phi)
 &=
 \lim_{\lambda\to\infty}\frac{\log h_{\Phi}(\lambda)}{\log\lambda}
 =
 \inf_{\lambda\in(1,\infty)}\frac{\log h_{\Phi}(\lambda)}{\log\lambda},
\end{align*}
respectively, with convention $\log\infty=\infty$.
\end{defn}

\begin{rem}\label{rem:index}
By the definition we see that $h_{\Phi}(1)=1$ 
and that $h_{\Phi}$ is increasing (i.e.\,non-decreasing) and submultiplicative
which means that
$h_{\Phi}(\lambda_1\lambda_2)\le h_{\Phi}(\lambda_1)h_{\Phi}(\lambda_2)$
for all $\lambda_1,\lambda_2\in(0,\infty)$.
In this case 
the above limits exist (permitting $\infty$)
and $0\le i(\Phi)\le I(\Phi)\le\infty$,
see \cite{Maligranda1985} for example.
If $\Phi\in\bdtwo$, then $a(\Phi)=0$ and $b(\Phi)=\infty$.
In this case $0<i(\Phi)\le I(\Phi)<\infty$,
see \cite{Gustavsson-Peetre1977,Maligranda1985} for example.
\end{rem}

\begin{rem}\label{rem:index 2}
Let $\Phi,\Psi\in\biP$ with $a(\Phi)=a(\Psi)=0$ and $b(\Phi)=b(\Psi)=\infty$.
\begin{enumerate}
\item 
If $\Phi\approx\Psi$, then $i(\Phi)=i(\Psi)$ and $I(\Phi)=I(\Psi)$.
\item 
If $\Phi\in\bcY$, then $1\le i(\Phi)\le I(\Phi)\le\infty$.
\item 
$\Phi\in\bntwo$ if and only if $1<i(\Phi)\le I(\Phi)\le\infty$.
\item 
$\Phi\in\bdtwo\cap\bntwo$ if and only if $1<i(\Phi)\le I(\Phi)<\infty$.
\item\label{item:bdtwo}
Let $\Phi\in\bcY$. Then $\Phi\in\bdtwo$ if and only if $1\le i(\Phi)\le I(\Phi)<\infty$.
\item
Let $0<i(\Phi)\le I(\Phi)<\infty$.
If $0<p<i(\Phi)\le I(\Phi)<q<\infty$, 
then there exists a positive constant $C$ such that,
for all $t,\lambda\in(0,\infty)$,
\begin{equation*}
 \Phi(\lambda t)
 \le
 C\max\left(\lambda^{p},\lambda^{q}\right)\Phi(t),
\end{equation*}
that is, 
$t\mapsto\dfrac{\Phi(t)}{t^p}$ is almost increasing
and
$t\mapsto\dfrac{\Phi(t)}{t^q}$ is almost decreasing.
\item\label{item:cY}
$\Phi\in\bcY$ if and only if 
$t\mapsto\dfrac{\Phi(t)}{t}$ is almost increasing
(\cite[Lemma~1.1.1]{Kokilashvili-Krbec1991}).
\end{enumerate}
\end{rem}

Next, 
we say that a function $\theta:\R^n\times(0,\infty)\to(0,\infty)$ 
satisfies the doubling condition if
there exists a positive constant $C$ such that,
for all $x\in\R^n$ and $r,s\in(0,\infty)$,
\begin{equation}\label{doubling}
 \frac1C\le\frac{\theta(x,r)}{\theta(x,s)}\le C,
 \quad\text{if} \ \ \frac12\le\frac{r}{s}\le2.
\end{equation}
We say that $\theta$ is almost increasing (resp. almost decreasing) if
there exists a positive constant $C$ such that, for all $x\in\R^n$ and $r,s\in(0,\infty)$,
\begin{equation*} 
 \theta(x,r)\le C\theta(x,s) \quad
 (\text{resp.}\ \theta(x,s)\le C\theta(x,r)),
 \quad\text{if $r<s$}.
\end{equation*}

For two functions $\theta,\kappa:\R^n\times(0,\infty)\to(0,\infty)$,
we write $\theta\sim\kappa$ if
there exists a positive constant $C$ such that,
for all $x\in\R^n$ and $r\in(0,\infty)$,
\begin{equation*}
 \frac1C\le\frac{\theta(x,r)}{\kappa(x,r)}\le C.
\end{equation*}

As same as Definition~\ref{defn:OM}
we also define $\LPp(\R^n,w)$ and $\wLPp(\R^n,w)$ 
by using generalized Young functions $\Phi\in\biPy$
together with $\|\cdot\|_{\Phi,\vp,w,B}$ and $\|\cdot\|_{\Phi,\vp,w,B,\weak}$, respectively.
Then $\|\cdot\|_{\LPp(\R^n,w)}$ and $\|\cdot\|_{\wLPp(\R^n,w)}$ are quasi norms
and thereby $\LPp(\R^n,w)$ and $\wLPp(\R^n,w)$ are quasi Banach spaces.

\begin{rem}\label{rem:approx norm}
Let $\Phi,\Psi\in\biPy$ and $\vp,\psi:\R^n\times(0,\infty)\to(0,\infty)$.
If $\Phi\approx\Psi$ and $\vp\sim\psi$,
then
$\LPp(\R^n,w)=\LPsps(\R^n,w)$ and $\wLPp(\R^n,w)=\wLPsps(\R^n,w)$
with equivalent quasi norms.
It is also known by \cite[Proposition~4.2]{Kawasumi-Nakai2020Hiroshima} that, 
for $\Phi\in\iPy$ and a measurable set $G$, 
\begin{equation}\label{weak type} 
 \sup_{t\in(0,\infty)}\Phi(t)w(G,f,t)
 =
 \sup_{t\in(0,\infty)}t\,w(G,\Phi(|f|),t).
\end{equation}
\end{rem}

In this paper we consider the following classes of $\vp$:
\begin{defn}\label{defn:cG}
For a weight $w$,
let $\cGdec_w$ be the set of all functions $\vp:\R^n\times(0,\infty)\to(0,\infty)$
such that 
$\vp$ is almost decreasing
and that
$r\mapsto\vp(x,r)w(B(x,r))$ is almost increasing.
That is,
there exists a positive constant $C$ such that, 
for all $x\in\R^n$ and $r,s\in(0,\infty)$,
\begin{equation*}
 C\vp(x,r)\ge \vp(x,s),
 \quad
 \vp(x,r)w(B(x,r))\le C\vp(x,s)w(B(x,s)), 
 \quad
 \text{if} \ r<s.
\end{equation*}
If $w(x)\equiv1$, we denote $\cGdec_w$ by $\cGdec$ simply.
\end{defn}

On the weights we consider the following Muckenhoupt $A_p$ classes:

\begin{defn}\label{defn:Ap}
For $p\in[1,\infty)$, 
let $A_p$ be the set of all weight functions $w$ such that
the following functional is finite: 
\begin{align*}
 [w]_{A_1}
 &=
 \sup_B\left(\frac1{|B|}\int_B w(x)\,dx\right)
       \|w^{-1}\|_{\Li(B)},
 &\quad\text{if} \ p=1,
\\
 [w]_{A_p}
 &=
 \sup_B\left(\frac1{|B|}\int_B w(x)\,dx\right)
       \left(\frac1{|B|}\int_B w(x)^{-1/(p-1)}\,dx\right)^{p-1},
 &\quad\text{if} \  p\in(1,\infty),
\end{align*}
where the suprema are taken over all balls $B$ in $\R^n$.
Let
\begin{equation*}
 A_{\infty}=\bigcup_{p\in[1,\infty)}A_p.
\end{equation*}
\end{defn}

Then the following properties are known:
Let $w$ is a weight.
Then $w\in A_{\infty}$
if and only if 
there exist positive constants $\delta$ and $C$ such that,
for any ball $B$ and its subset $E$,
\begin{equation}\label{delta}
 \frac{w(E)}{w(B)}\le C\left(\frac{|E|}{|B|}\right)^{\delta}.
\end{equation}
If $1\le p<q\le\infty$, then $A_p\subset A_q$.
Let $p\in(1,\infty)$. If $w\in A_p$, then $w\in A_r$ for some $r\in[1,p)$.

Let $w\in A_p$ for some $p\in[1,\infty)$. 
Then, for any ball $B$,
\begin{equation}\label{fint B}
 \left(\frac1{|B|}\int_B|f(x)|\,dx\right)^p
 \le
 [w]_{A_p}
 \frac1{w(B)}\int_B|f(x)|^p w(x)\,dx.
\end{equation}
Moreover, there exists a positive constant $C$ such that,
for any ball $B$ and $k\in(1,\infty)$,
\begin{equation}\label{w kB}
 w(kB)\le C k^{np}[w]_{A_p}w(B).
\end{equation}

If $w\in A_p$ for some $p\in[1,\infty)$ and $\vp\in\cGdec_w$, 
then $\vp$ satisfies the doubling condition \eqref{doubling},
since $w$ satisfies \eqref{w kB}.

For the properties of $A_p$-weights, 
see \cite{GarciaCuerva-RubiodeFrancia1985,Grafakos2014GTM249}
for example.

\section{Main results}\label{sec:main}

The Hardy-Littlewood maximal operator is defined by
\begin{equation*}
 Mf(x)=\sup_{B\ni x}\frac1{|B|}\int_B |f(y)|\,dy,
\end{equation*}
for locally integrable functions $f$,
where the supremum is taken over all balls $B$ containing $x$.
It is known that, if $\Phi\in\biPy$ and $\vp\in\cGdec$,
then the Hardy-Littlewood maximal operator $M$ is bounded 
from $\LPp(\R^n)$ to $\wLPp(\R^n)$.
Moreover, if $\Phi\in\bntwo$,
then $M$ is bounded 
from $\LPp(\R^n)$ to itself and 
from $\wLPp(\R^n)$ to itself,
see \cite{Kawasumi-Nakai-Shi2021MathNachr,Nakai2008Studia}.


Next we state known results for the boundedness of the Calder\'on-Zygmund operator.
First we recall its definition following \cite{Yabuta1985}.
Let $\cS(\R^n)$ be the set of all Schwartz functions on $\R^n$
and $\cS'(\R^n)$ be the dual spaces of $\cS(\R^n)$.
Let $\Omega$ be the set of all increasing functions $\omega:(0,\infty)\to(0,\infty)$
such that 
$\int_0^1\frac{\omega(t)}{t}dt<\infty$. 

\begin{defn}[{standard kernel}]\label{defn:Kernel}
Let $\omega\in \Omega$. 
A continuous function $K(x,y)$ on $\R^n\times\R^n\setminus\{(x,x)\in\R^{2n}\}$
is said to be a standard kernel of type $\omega$ 
if the following conditions are satisfied;
\begin{gather*} 
     |K(x,y)|\le \frac{C}{|x-y|^n} 
     \quad\text{for}\quad x\not=y,
\\     
  \begin{split}
     |K(x,y)-K(x,z)|+|K(y,x)-K(z,x)|
     \le 
     \frac{C}{|x-y|^{n}} 
     \,\omega\!\left(\frac{|y-z|}{|x-y|}\right) & \\
     \text{for}\quad  2|y-z|<|x-y|. &
  \end{split}
\end{gather*} 
\end{defn}

\begin{defn}[{Calder\'on-Zygmund operator}]\label{defn:CZO}
Let $\omega\in \Omega$.
A linear operator $T$ from $\cS(\R^n)$ to $\cS'(\R^n)$
is said to be a Calder\'on-Zygmund operator of type $\omega$,
if $T$ is bounded on $L^2(\R^n)$
and there exists a standard kernel $K$ of type $\omega$ such that,
for $f\in \Cic(\R^n)$,
\begin{equation} 
     Tf(x)=\int_{\R^n} K(x,y)f(y)\,dy, \quad x\notin\supp f.
                                                  \label{CZ3}
\end{equation} 
\end{defn}

\begin{rem}\label{rem:CZO}
If $x\notin\supp f$, then $K(x,y)$ is bounded on $\supp f$ with respect to $y$.
Therefore,
if \eqref{CZ3} holds for $f\in \Cic(\R^n)$,
then \eqref{CZ3} holds for $f\in L^1_{\comp}(\R^n)$.
\end{rem}

It is known by \cite{Yabuta1985} that 
any Calder\'on-Zygmund operator of type $\omega\in \Omega$
is bounded on $L^p(\R^n)$ for $1<p<\infty$.
This result was extended to Orlicz-Morrey spaces $\LPp(\R^n)$
by \cite{Nakai2008KIT}
as the following:
Let $\vp:(0,\infty)\to(0,\infty)$.
Assume that $\vp\in\cGdec$ 
and that
there exists a positive constant $C$ such that,
for all $r\in(0,\infty)$,
\begin{equation*} 
 \int_r^{\infty}\frac{\vp(t)}{t}\,dt\le C\vp(r).
\end{equation*}
Let $\Phi\in\Delta_2\cap\nabla_2$.
For $f\in\LPp(\R^n)$, we define $Tf$ on each ball $B$ by
\begin{equation*} 
 Tf(x)=T(f\chi_{2B})(x)+\int_{\R^n\setminus 2B}K(x,y)f(y)\,dy,
 \quad x\in B.
\end{equation*}
Then the first term in the right hand side 
is well defined, 
since $f\chi_{2B}\in\LP_{\comp}(\R^n)\subset L^1_{\comp}(\R^n)$,
and the integral of the second term converges absolutely.
Moreover, 
$Tf(x)$ is independent of the choice of the ball $B$ containing $x$.
By this definition we can show that $T$ is a bounded operator from $\LPp(\R^n)$ to itself.
For the weighted boundedness, it is also known by \cite{Yabuta1985} that, 
if $w\in A_1$, 
then $T$ is bounded from $L^1(\R^n,w)$ to $\wL^1(\R^n,w)$,
and, if $w\in A_p$, $1<p<\infty$, 
then $T$ is bounded from $L^p(\R^n,w)$ to itself.

In this paper we extend the above results to the weighted Orlicz-Morrey space and
its weak version.
As a corollary we also get the boundedness of $T$ from $\wL^p(\R^n,w)$ to itself
if $w\in A_p$, $1<p<\infty$.
The main result is the following:

\begin{thm}\label{thm:M CZO}
Let $M$ be the Hardy-Littlewood maximal operator,
and let $T$ be a Calder\'on-Zygmund operator of type $\omega\in\Omega$.
Let $\Phi\in\bcY$, $w\in A_{i(\Phi)}$ and $\vp\in\cGdec_w$.
\begin{enumerate}
\item 
If $i(\Phi)=1$, then
$M$ is bounded from $\LPp(\R^n,w)$ to $\wLPp(\R^n,w)$. 
If $1<i(\Phi)\le\infty$, then
$M$ is bounded from $\LPp(\R^n,w)$ to itself and from $\wLPp(\R^n,w)$ to itself. 
\item
Assume that 
there exists a positive constant $C$ such that,
for all $x\in\R^n$ and $r\in(0,\infty)$,
\begin{equation}\label{int vp x}
 \int_r^{\infty}\frac{\vp(x,t)}{t}\,dt\le C\vp(x,r).
\end{equation}
If $i(\Phi)=1\le I(\Phi)<\infty$, then
$T$ is bounded from $\LPp(\R^n,w)$ to $\wLPp(\R^n,w)$. 
If $1<i(\Phi)\le I(\Phi)<\infty$, then
$T$ is bounded from $\LPp(\R^n,w)$ to itself and from $\wLPp(\R^n,w)$ to itself. 
\end{enumerate}
\end{thm}

Ho~\cite{Ho2013} proved the boundedness of $M$ on $\LPp(\R^n,w)$ under stronger conditions.
He treated the vector valued inequality.

To prove Theorem~\ref{thm:M CZO}
we need the modular inequalities 
for which the assumption $w\in A_{i(\Phi)}$ is necessary,
see Corollary~\ref{cor:modular}.

From the theorem above,
for the operators $M$ and $T$, we get the following corollaries immediately:

\begin{cor}\label{cor:M CZO Orl}
Let $\Phi\in\bcY$, $w\in A_{i(\Phi)}$ and $\vp\in\cGdec_w$.
\begin{enumerate}
\item 
If $i(\Phi)=1$, then
$M$ is bounded from $\LP(\R^n,w)$ to $\wLP(\R^n,w)$. 
If $1<i(\Phi)\le\infty$, then
$M$ is bounded from $\LP(\R^n,w)$ to itself and from $\wLP(\R^n,w)$ to itself. 
\item
Assume that $\vp$ satisfies \eqref{int vp x}.
If $i(\Phi)=1\le I(\Phi)<\infty$, then
$T$ is bounded from $\LP(\R^n,w)$ to $\wLP(\R^n,w)$. 
If $1<i(\Phi)\le I(\Phi)<\infty$, then
$T$ is bounded from $\LP(\R^n,w)$ to itself and from $\wLP(\R^n,w)$ to itself. 
\end{enumerate}
\end{cor}

\begin{cor}\label{cor:M CZO Mor}
Let $p\in[1,\infty)$, $w\in A_p$ and $\vp\in\cGdec_w$.
\begin{enumerate}
\item 
If $p=1$, then
$M$ is bounded from $L^{(1,\vp)}(\R^n,w)$ to $\wL^{(1,\vp)}(\R^n,w)$. 
If $1<p<\infty$, then
$M$ is bounded from $L^{(p,\vp)}(\R^n,w)$ to itself and from $\wL^{(p,\vp)}(\R^n,w)$ to itself. 
\item
Assume that $\vp$ satisfies \eqref{int vp x}.
Then $T$ has the same boundedness as $M$.
\end{enumerate}
\end{cor}

Let $w\in A_p$ for some $p\in[1,\infty)$.
If $\vp(B)=w(B)^{\kappa-1}$ for some $\kappa\in[0,1)$,
then $\vp(kB)\ls k^{-n\delta(1-\kappa)}\vp(B)$ for some $\delta>0$ and all $k\ge1$ by \eqref{delta}.
Hence, $\vp$ satisfies \eqref{int vp x}.
Then we also have the following corollary:

\begin{cor}\label{cor:M CZO Mk}
If $w\in A_1$ and $\vp\in\cGdec_w$,
then
both $M$ and $T$ are bounded from $L^{1,\kappa}(\R^n,w)$ to $\wL^{1,\kappa}(\R^n,w)$. 
If $1<p<\infty$, $w\in A_p$ and $\vp\in\cGdec_w$, then
both $M$ and $T$ are bounded from $L^{p,\kappa}(\R^n,w)$ to itself and from $\wL^{p,\kappa}(\R^n,w)$ to itself. 
\end{cor}

\section{Properties on Young functions}\label{sec:prop}

In this section we state the properties of Young functions and their generalization. 
For the theory of Orlicz spaces, 
see \cite{Kita2009,Krasnoselsky-Rutitsky1961,Maligranda1989} for example.

For $\Phi\in\biP$,
we recall the generalized inverse of $\Phi$
in the sense of O'Neil \cite[Definition~1.2]{ONeil1965}.

\begin{defn}\label{defn:ginverse}
For $\Phi\in\biP$ and $u\in[0,\infty]$, let
\begin{equation*} 
 \Phi^{-1}(u)
 = 
\begin{cases}
 \inf\{t\ge0: \Phi(t)>u\}, & u\in[0,\infty), \\
 \infty, & u=\infty.
\end{cases}
\end{equation*}
\end{defn}

Let $\Phi\in\biP$. 
Then $\Phi^{-1}$ is finite, increasing and right continuous on $[0,\infty)$
and positive on $(0,\infty)$.
If $\Phi$ is bijective from $[0,\infty]$ to itself,
then $\Phi^{-1}$ is the usual inverse function of $\Phi$.
In general, if $\Phi\in\biP$, then
\begin{equation*} 
 \Phi(\Phi^{-1}(u)) \le u \le  \Phi^{-1}(\Phi(u))
 \quad\text{for all $u\in[0,\infty]$},
\end{equation*}
which is a generalization of Property 1.3 in \cite{ONeil1965},
see \cite[Proposition~2.2]{Shi-Arai-Nakai2019Taiwan}.
Let $\Phi, \Psi\in\biP$.
Then
\begin{equation*} 
     \Phi(C^{-1}t)\le\Psi(t)\le\Phi(Ct)
     \quad\text{for all}\ t\in[0,\infty],
\end{equation*} 
if and only if 
\begin{equation*} 
     C^{-1}\Phi^{-1}(t)\le\Psi^{-1}(t)\le C\Phi^{-1}(t)
     \quad\text{for all}\ t\in[0,\infty],
\end{equation*} 
see \cite[Lemma~2.3]{Shi-Arai-Nakai2019Taiwan}.
That is,
$\Phi\approx\Psi$ if and only if $\Phi^{-1}\sim\Psi^{-1}$.

\begin{defn}\label{defn:complem}
For a Young function $\Phi$, 
its complementary function is defined by
\begin{equation*}
\cPhi(t)= 
\begin{cases}
   \sup\{tu-\Phi(u):u\in[0,\infty)\}, & t\in[0,\infty), \\
   \infty, & t=\infty.
 \end{cases}
\end{equation*}
\end{defn}

Then $\cPhi$ is also a Young function,
and $(\Phi,\cPhi)$ is called a complementary pair.
For example, 
if $\Phi(t)=t^p/p$, then $\cPhi(t)=t^{p'}/p'$
for $p,p'\in(1,\infty)$ and $1/p+1/p'=1$.
If $\Phi(t)=t$, then
\begin{equation*}
 \cPhi(t)=
\begin{cases}
 0, & t\in[0,1], \\
 \infty, & t\in(1,\infty].
\end{cases}
\end{equation*}
Namely, $\cPhi$ is not necessary in $\cY$ even if $\Phi\in\cY$.

Let $(\Phi,\cPhi)$ be a complementary pair of Young functions. 
Then the following inequality holds (\cite[(1.3)]{Torchinsky1976}):
\begin{equation}\label{Phi cPhi r}
 t\le\Phi^{-1}(t) \cPhi^{-1}(t)\le2t
 \quad\text{for}\quad t\in[0,\infty].
\end{equation}

Let $\Phi$ be a Young function and $(X,\mu)$ a measure space, 
and let $\LP(X,\mu)$ be the Orlicz space with the norm $\|\cdot\|_{\LP(X,\,\mu)}$.
Then a simple calculation shows that,
for any measurable subset $G\subset X$ with $\mu(G)>0$,
\begin{equation}\label{chi Orlicz norm}
 \|\chi_G\|_{\LP(X,\,\mu)}
 =
 \frac1{\Phi^{-1}(1/\mu(G))}.
\end{equation}
Let $(\Phi,\cPhi)$ be a complementary pair of Young functions. 
Then the following generalized H\"older's inequality holds (see \cite{ONeil1965}):
\begin{equation}\label{g Holder}
 \int_{X} |f(x)g(x)| \,d\mu(x) \le 2\|f\|_{\LP(X,\,\mu)} \|g\|_{\LcP(X,\,\mu)}.
\end{equation}

Let $\Phi\in\iPy$, $\vp:\R^n\times(0,\infty)\to(0,\infty)$ and $B=B(a,r)\subset\R^n$,
and let $\mu_B=w\,dx/(\vp(B)w(B))$.
Then 
by the definition of $\|\cdot\|_{\Phi,\vp,w,B}$ and \eqref{chi Orlicz norm}
we have
\begin{equation}\label{chi norm B}
 \|\chi_B\|_{\Phi,\vp,w,B}
 =
 \|\chi_B\|_{\LP(B,\mu_B)}
 =
 \frac1{\Phi^{-1}(1/\mu_B(B))}
 =
 \frac1{\Phi^{-1}(\vp(B))}.
\end{equation}
Moreover, by \eqref{g Holder} we have
\begin{equation}\label{g Holder B}
 \frac1{\vp(B)w(B)}\int_{B} |f(x)g(x)| w(x)\,dx \le 2\|f\|_{\Phi,\vp,w,B} \|g\|_{\cPhi,\vp,w,B}.
\end{equation}

Here we show the following lemma:

\begin{lem}\label{lem:p}
Let $w$ be a weight, $\Phi\in\bcY$ and $\vp:\R^n\times(0,\infty)\to(0,\infty)$.
Then there exists a positive constant $C$ such that, for all balls $B$,
\begin{equation}\label{fint 1}
 \frac1{w(B)}\int_B |f(x)|w(x)\,dx
 \le
 C\Phi^{-1}(\vp(B))\|f\|_{\Phi,\vp,w,B}.
\end{equation}
Moreover, 
assume that $t\mapsto\Phi(t)/t^p$ is almost increasing for some $p\in(1,\infty)$.
Then
there exists a positive constant $C_p$ such that
\begin{equation}\label{fint p}
 \left(\frac1{w(B)}\int_{B} |f(x)|^p w(x)\,dy\right)^{1/p}
 \le
 C_p\Phi^{-1}(\vp(B)) \|f\|_{\Phi,\vp,w,B},
\end{equation}
and, for all $q\in[1,p)$,
there exists a positive constant $C_{p,q}$ such that
\begin{equation}\label{fint q}
 \left(\frac1{w(B)}\int_{B} |f(x)|^q w(x)\,dy\right)^{1/q}
 \le
 C_{p,q}\Phi^{-1}(\vp(B)) \|f\|_{\Phi,\vp,w,B,\weak}.
\end{equation}
\end{lem}

\begin{proof}
We may assume that $\Phi\in\cY$.
By \eqref{g Holder B}, \eqref{chi norm B} and \eqref{Phi cPhi r}
we have
\begin{align*}
 \frac1{w(B)}\int_B |f(x)|w(x)\,dx
 &\le
 2\vp(B)\|f\|_{\Phi,\vp,w,B}\|\chi_B\|_{\cPhi,\vp,w,B} 
\\
 &=
 \frac{2\vp(B)}{\cPhi^{-1}(\vp(B))}\|f\|_{\Phi,\vp,w,B} 
\\
 &\le
 2\Phi^{-1}(\vp(B))\|f\|_{\Phi,\vp,w,B}.
\end{align*}

Next, we assume that $t\mapsto\Phi(t)/t^p$ is almost increasing for some $p\in(1,\infty)$.
Then
$t\mapsto\Phi(t^{1/p})/t$ is almost increasing,
which implies $\Phi((\cdot)^{1/p})\in\bcY$, see Remark~\ref{rem:index 2}.
Let $\Phi_{p}\in\cY$ such that
$\Phi_{p}\approx\Phi\left((\cdot)^{1/p}\right)$.
Then ${\Phi_{p}}^{-1} \sim (\Phi^{-1})^{p}$ 
and
$\||f|^p\|_{\Phi_{p},\vp,w,B}\sim(\|f\|_{\Phi,\vp,w,B})^p$.
Using $\eqref{fint 1}$, we have 
\begin{align*}
 \left(\frac1{w(B)}\int_{B} |f(x)|^p w(x)\,dx\right)^{1/p}
 &\ls 
 \big({\Phi_{p}}^{-1}(\vp(B)) \||f|^p\|_{\Phi_{p},\vp,w,B}\big)^{1/p}
\\
 &\sim
 \Phi^{-1}(\vp(B)) \|f\|_{\Phi,\vp,w,B}.
\end{align*}

Finally, we show \eqref{fint q}.
We may assume that $\|f\|_{\Phi,\vp,w,B,\weak}=1$.
Then 
\begin{equation*}
 w(B,f,t) \le \frac{\vp(B)w(B)}{\Phi(t)} 
 \quad\text{for all}\quad t\in(0,\infty).
\end{equation*}
Let $q\in[1,p)$ and
$t_0=\Phi^{-1}(\vp(B))$.
Then $\Phi(t_0)=\vp(B)$. 
Since $t\mapsto\Phi(t)/t^p$ is almost increasing,
\begin{align*}
 \int_B|f(x)|^qw(x)\,dx
 &=
 q\int_0^{t_0} t^{q-1} w(B,f,t) \,dt
 +q\int_{t_0}^{\infty} t^{q-1}w(B,f,t) \,dt
\\
 &\le
 {t_0}^qw(B)
 + q\int_{t_0}^{\infty} t^{q-1} \frac{\vp(B)w(B)}{\Phi(t)} \,dt
\\
 &=
 {t_0}^qw(B)
 + q\vp(B)w(B) \int_{t_0}^{\infty} \frac{ t^p }{ \Phi(t) } t^{-p+q-1} \,dt
\\
 &\ls
 {t_0}^qw(B)
 + q\vp(B)w(B) \frac{{t_0}^p}{ \Phi({t_0}) } \int_{t_0}^{\infty} t^{-p+q-1} \,dt
\\
 &=
 {t_0}^qw(B)
 +\frac{q}{p-q}{t_0}^qw(B). 
\end{align*}
This shows the conclusion.
\end{proof}

At the end of this section we state another lemma.

\begin{lem}[{\cite[Lemma~4.4]{Shi-Arai-Nakai2021Banach}}]\label{lem:int Phi vp}
Let $\Phi\in\dtwo$ and $\vp:\R^n\times(0,\infty)\to(0,\infty)$.
If $\vp$ satisfies \eqref{int vp x},
then
there exists a positive constant $C$ such that,
for all $x\in\R^n$ and $r\in(0,\infty)$,
\begin{equation*} 
 \int_r^{\infty}\frac{\Phi^{-1}(\vp(x,t))}{t}\,dt\le C\Phi^{-1}(\vp(x,r)).
\end{equation*}
\end{lem}

Note that \cite[Lemma~4.4]{Shi-Arai-Nakai2021Banach}
is the case $\vp:(0,\infty)\to(0,\infty)$.
However the proof is the same.

\section{Modular inequalities}\label{sec:modular}

In this section we show 
the modular inequalities with $\Phi\in\bcY$
by using the indices $i(\Phi)$ and $I(\Phi)$.

We first state known weighted inequalities.

\begin{thm}[\cite{Coifman1972,Coifman-Fefferman1974,Hunt-Muckenhoupt-Wheeden1973,Muckenhoupt1972,Yabuta1985}]\label{thm:Lp bdd}
Let $M$ be the Hardy-Littlewood maximal operator, 
and let $T$ be a Calder\'on-Zygmund operator of type $\omega\in\Omega$.
Let $w\in A_p$, $1\le p\le\infty$.
\begin{enumerate}
\item 
If $1<p\le\infty$, then
\begin{equation*}
 \int_{\R^n}(Mf(x))^pw(x)\,dx
 \le
 C\int_{\R^n}|f(x)|^pw(x)\,dx.
\end{equation*}
If $p=1$, then
\begin{equation*}
 \sup_{t\in(0,\infty)}tw(Mf,t)
 \le
 C\int_{\R^n}|f(x)|w(x)\,dx.
\end{equation*}
\item
If $1<p<\infty$, then
\begin{equation*}
 \int_{\R^n}|Tf(x)|^pw(x)\,dx
 \le
 C\int_{\R^n}|f(x)|^pw(x)\,dx.
\end{equation*}
If $p=1$, then
\begin{equation*}
 \sup_{t\in(0,\infty)}tw(Tf,t)
 \le
 C\int_{\R^n}|f(x)|w(x)\,dx.
\end{equation*}
\end{enumerate}
\end{thm}

Coifman and Fefferman~\cite{Coifman-Fefferman1974} prove the inequality
\begin{equation}\label{CZ M}
 \int_{\R^n}|Tf(x)|^pw(x)\,dx
 \le
 C\int_{\R^n}(Mf(x))^pw(x)\,dx,
\end{equation}
for any $w\in A_{\infty}$ and any Calder\'on-Zygmund operator with standard kernel
(the case $\omega(t)=t$ in Definition~\ref{defn:Kernel}).
By the kernel estimates in \cite{Yabuta1985}
we see that the inequality \eqref{CZ M} valids for any Calder\'on-Zygmund operator 
of type $\omega\in\Omega$.
From the inequality \eqref{CZ M} 
Curbera, Garcia-Cuerva, Martell and Perez~\cite{Curbera-GarciaCuerva-Martell-Perez2006}
proved the following inequalities:
\begin{align}\label{modular T M}
 \int_{\R^n}\Phi(|Tf(x)|)w(x)\,dx
 &\le
 C\int_{\R^n}\Phi(Mf(x))w(x)\,dx,
\\
\label{ww modular T M}
 \sup_{t\in(0,\infty)}\Phi(t)w(Tf,t)
 &\le
 C\sup_{t\in(0,\infty)}\Phi(t)w(Mf,t).
\end{align}
Then they 
proved the following modular inequalities
except \eqref{ww modular M} and \eqref{ww modular CZO}, 
see \cite[Theorem~3.7]{Curbera-GarciaCuerva-Martell-Perez2006}.
In this section we prove \eqref{ww modular M} and then \eqref{ww modular CZO}. 
That is, we have the following theorem:

\begin{thm}%
\label{thm:modular}
Let $M$ be the Hardy-Littlewood maximal operator, 
and let $T$ be a Calder\'on-Zygmund operator of type $\omega\in\Omega$.
Let $\Phi\in\bcY$, and let $w\in A_{i(\Phi)}$.
\begin{enumerate}
\item 
If $1<i(\Phi)\le\infty$, then
\begin{align}\label{modular M}
 \int_{\R^n}\Phi(Mf(x))w(x)\,dx
 &\le
 C\int_{\R^n}\Phi(C|f(x)|)w(x)\,dx,
\\
\label{ww modular M}
 \sup_{t\in(0,\infty)}\Phi(t)w(Mf,t)
 &\le
 C\sup_{t\in(0,\infty)}\Phi(t)w(Cf,t).
\end{align}
If $i(\Phi)=1$, then
\begin{equation}\label{w modular M}
 \sup_{t\in(0,\infty)}\Phi(t)w(Mf,t)
 \le
 C\int_{\R^n}\Phi(C|f(x)|)w(x)\,dx.
\end{equation}
\item
If $1<i(\Phi)\le I(\Phi)<\infty$, then
\begin{align}\label{modular CZO}
 \int_{\R^n}\Phi(|Tf(x)|)w(x)\,dx
 &\le
 C\int_{\R^n}\Phi(C|f(x)|)w(x)\,dx,
\\
\label{ww modular CZO}
 \sup_{t\in(0,\infty)}\Phi(t)w(Tf,t)
 &\le
 C\sup_{t\in(0,\infty)}\Phi(t)w(Cf,t).
\end{align}
If $i(\Phi)=1\le I(\Phi)<\infty$, then
\begin{equation}\label{w modular CZO}
 \sup_{t\in(0,\infty)}\Phi(t)w(Tf,t)
 \le
 C\int_{\R^n}\Phi(C|f(x)|)w(x)\,dx.
\end{equation}
\end{enumerate}
\end{thm}

Kokilashvili and Krbec~\cite{Kokilashvili-Krbec1991} also investigated the modular inequalities 
except \eqref{ww modular M} and \eqref{ww modular CZO}.
If $1<i(\Phi)\le I(\Phi)<\infty$ and $w$ is a weight, 
then the modular inequality \eqref{modular M} 
implies $w\in A_{i(\Phi)}$.
see \cite[Theorem~2.1.1]{Kokilashvili-Krbec1991}.
If $T=R_j$, $i=1,\dots, n$, which are the Reisz transforms,
then \eqref{w modular CZO} also 
implies $w\in A_{i(\Phi)}$, 
see \cite[Theorem~3.1.1]{Kokilashvili-Krbec1991}.
From this fact, \eqref{modular T M} and \eqref{ww modular T M} we have the following corollary:

\begin{cor}\label{cor:modular}
Let $M$ be the Hardy-Littlewood maximal operator, 
and let $R_j$, $i=1,\dots, n$, be the Reisz transforms.
Let $w$ be a weight and 
$\Phi\in\bdtwo\cap\bntwo$, i.e., $1<i(\Phi)\le I(\Phi)<\infty$.
Then the following are equivalent:
\begin{enumerate}
\item \ 
$\displaystyle
 \int_{\R^n}\Phi(Mf(x))w(x)\,dx
 \le
 C\int_{\R^n}\Phi(C|f(x)|)w(x)\,dx,
$
\item \
$\displaystyle
 \sup_{t\in(0,\infty)}\Phi(t)w(Mf,t)
 \le
 C\sup_{t\in(0,\infty)}\Phi(t)w(Cf,t),
$
\item \
$\displaystyle
 \sup_{t\in(0,\infty)}\Phi(t)w(Mf,t)
 \le
 C\int_{\R^n}\Phi(C|f(x)|)w(x)\,dx,
$
\item \
$\displaystyle
 \int_{\R^n}\Phi(|R_jf(x)|)w(x)\,dx
 \le
 C\int_{\R^n}\Phi(C|f(x)|)w(x)\,dx,
$
\item \
$\displaystyle
 \sup_{t\in(0,\infty)}\Phi(t)w(R_jf,t)
 \le
 C\sup_{t\in(0,\infty)}\Phi(t)w(Cf,t),
$
\item \
$\displaystyle
 \sup_{t\in(0,\infty)}\Phi(t)w(R_jf,t)
 \le
 C\int_{\R^n}\Phi(C|f(x)|)w(x)\,dx,
$
\item \
$\displaystyle
 w \in A_{i(\Phi)}.
$
\end{enumerate}
\end{cor}

Note that another pair of indices $a_{\Phi}$ and $b_{\Phi}$ are defined by
\begin{equation*}
 a_{\Phi}=\inf_{t\in(0,\infty)}\frac{t\Phi'(t)}{\Phi(t)},
 \quad
 b_{\Phi}=\sup_{t\in(0,\infty)}\frac{t\Phi'(t)}{\Phi(t)}.
\end{equation*}
Then
$t\mapsto\Phi(t)/t^{a_{\Phi}}$ is increasing and 
$t\mapsto\Phi(t)/t^{b_{\Phi}}$ is decreasing (not almost),
see \cite[Proposition~2.1~(ii) and (iii)]{Fu-Yang-Yuan2012} for example. 
However, these indices $a_{\Phi}$ and $b_{\Phi}$ are not sharp for the modular inequalities, 
see the following example.

\begin{exmp}\label{exmp:type}
Let
\begin{equation*}
 \Phi(t)=
 \begin{cases}
  t^2, & t\in[0,1/4], \\
  t/2-1/{16}, & t\in(1/4,1/2], \\
  t^2/2+1/{16}, & t\in(1/2,\infty).
 \end{cases}
\end{equation*}
Then
\begin{equation*}
 i(\Phi)=I(\Phi)=2, \quad\text{but}\quad a_{\Phi}=4/3, \ b_{\Phi}=2.
\end{equation*}
\end{exmp}

Liu and Wang~\cite{Liu-Wang2013} also considered
the weighted Orlicz spaces and they showed 
the modular inequality \eqref{ww modular M}
by using the Marcinkiewicz-type interpolation theorem,
see the proof of \cite[Theorem~5.1]{Liu-Wang2013}.
However, they 
used indices $a_{\Phi}$ and $b_{\Phi}$,
which are not sharp as shown by Example~\ref{exmp:type}.

To prove \eqref{ww modular M} 
we prepare the following lemma:

\begin{lem}\label{lem:Mw}
For $w\in A_{\infty}$, 
let 
\begin{equation*}
 M_wf(x)=\sup_{B\ni x}\frac1{w(B)}\int_B |f(y)|w(y)\,dy.
\end{equation*}
Let $\Phi\in\bcY$.
If $i(\Phi)>1$, then there exists a positive constant $c_1$ such that
\begin{equation*}
 \sup_{t\in(0,\infty)}\Phi(t)w(M_wf,t)
 \le
 c_1\sup_{t\in(0,\infty)}\Phi(t)w(c_1f,t).
\end{equation*}
\end{lem}

\begin{proof}
We may asssume that $\Phi\in\cY$.
First note that $M_w$ is bounded from $\wL^p(\R^n,w)$ to itself 
as same as $M$ is bounded from $\wL^p(\R^n)$ to itself if $p\in(1,\infty]$.
If $i(\Phi)>1$, then $\Phi^{\theta}\in\biPy$ for some $\theta\in(0,1)$.
In this case we have the inequality
\begin{equation*}
 \Phi(M_wf(x))\le (cM_w(\Phi(c|f|)^{\theta})(x))^{1/\theta},
\end{equation*}
for some constant $c$ 
by the same way as \cite[Proof of Proposition~5.1]{Curbera-GarciaCuerva-Martell-Perez2006}.
Then
\begin{align*}
 \sup_{t\in(0,\infty)}tw(\Phi(M_wf),t) 
 &\le
 \sup_{t\in(0,\infty)}tw((cM_w(\Phi(c|f|)^{\theta}))^{1/\theta},t) \\
 &=
 \sup_{t\in(0,\infty)}t^{1/\theta}w(cM_w(\Phi(c|f|)^{\theta}),t) \\
 &\ls
 \sup_{t\in(0,\infty)}t^{1/\theta}w(\Phi(c|f|)^{\theta},t) \\
 &=
 \sup_{t\in(0,\infty)}tw(\Phi(c|f|),t).
\end{align*}
By \eqref{weak type} 
we have the conclusion.
\end{proof}

\begin{proof}[Proof of \eqref{ww modular M}]
We may asssume that $\Phi\in\cY$.
We use a similar way to the proof of \eqref{modular M} in 
\cite{Curbera-GarciaCuerva-Martell-Perez2006}.
Let $w\in A_i(\Phi)$.
In both cases $1<i(\Phi)<\infty$ and $i(\Phi)=\infty$,
there exists $r\in(1,i(\Phi))$ such that $w\in A_r$.
Set $\Phi_r(t)=\Phi(t^{1/r})$.
Then $i(\Phi_r)=i(\Phi)/r>1$.
By \eqref{fint B} 
we have 
$Mf(x)\le\left([w]_{A_r}M_w(|f|^r)(x)\right)^{1/r}$
and then 
\begin{equation*}
 \Phi(Mf(x))\le\Phi_r(M_w\widetilde{f}(x)),
\end{equation*}
where $\widetilde{f}=[w]_{A_r}|f(x)|^r$.
By Lemma~\ref{lem:Mw} and \eqref{weak type}
we have
\begin{align*}
 \sup_{t\in(0,\infty)}tw(\Phi(Mf),t) 
 &\le
 \sup_{t\in(0,\infty)}tw(\Phi_r(M_w\widetilde{f}),t) \\
 &\ls
 \sup_{t\in(0,\infty)}tw(\Phi_r(c_1\widetilde{f}),t) \\
 &=
 \sup_{t\in(0,\infty)}tw(\Phi(Cf),t),
\end{align*}
which shows the conclusion.
\end{proof}

\section{Proofs}\label{sec:proof}

To prove Theorem~\ref{thm:M CZO}, we prepare three lemmas.

\begin{lem}\label{lem:Mf1}
Let $\Phi\in\bcY$, $w\in A_{i(\Phi)}$ and $\vp\in\cGdec_w$.
Let $B$ be a ball.
If $i(\Phi)=1$ and $\|f\|_{\LPp(\R^n,w)}=1$,
then
\begin{equation*}
 \|M(f\chi_{2B})\|_{\Phi,\vp,w,B,\weak}\le C
 \quad\text{and}\quad
 \|T(f\chi_{2B})\|_{\Phi,\vp,w,B,\weak}\le C.
\end{equation*}
If $1<i(\Phi)\le I(\Phi)\le\infty$ and $\|f\|_{\LPp(\R^n,w)}=1$ or $\|f\|_{\wLPp(\R^n,w)}=1$, 
then
\begin{equation*}
 \|M(f\chi_{2B})\|_{\Phi,\vp,w,B}\le C
 \quad\text{or}\quad
 \|M(f\chi_{2B})\|_{\Phi,\vp,w,B, \weak}\le C,
\end{equation*}
respectively.
If $1<i(\Phi)\le I(\Phi)<\infty$ and $\|f\|_{\LPp(\R^n,w)}=1$ or $\|f\|_{\wLPp(\R^n,w)}=1$, 
then
\begin{equation*}
 \|T(f\chi_{2B})\|_{\Phi,\vp,w,B}\le C
 \quad\text{or}\quad
 \|T(f\chi_{2B})\|_{\Phi,\vp,w,B, \weak}\le C,
\end{equation*}
respectively.
In the above the constant $C$ is independent of $f$ and $B$.
\end{lem}

\begin{proof}
We use Theorem~\ref{thm:modular}.
We only prove the case $i(\Phi)=1$ and $M$, since the other cases are similar.
If $i(\Phi)=1$ and $\|f\|_{\LPp(\R^n,w)}=1$,
then by \eqref{w modular M} we have 
\begin{align*}
 \sup_{t\in(0,\infty)}\Phi(t)\,w\left(B,{M(f\chi_{2B})}/C,t\right)
 &\le 
 \sup_{t\in(0,\infty)}\Phi(t)\,w\left({M(f\chi_{2B})}/C,t\right)
\\
 &\le C \int_{2B}\Phi(|f|)w(x)\,dx
\\
 &\le C \vp(2B)w(2B)
 \le C' \vp(B)w(B).
\end{align*}
We may assume that $C'\ge1$. 
Then
\begin{equation*}
 \sup_{t\in(0,\infty)}\Phi(t)\,w\left(B,M(f\chi_{2B})/(C'C),t\right)
 \le
 \vp(B)w(B),
\end{equation*}
which shows the conclusion.
\end{proof}

\begin{lem}\label{lem:Mf2}
Let $\Phi\in\bcY$, $w\in A_{i(\Phi)}$ and $\vp\in\cGdec_w$.
Let $B$ be a ball.
If one of the following three conditions holds;
{\rm (1)} $i(\Phi)=1$ and $\|f\|_{\LPp(\R^n,w)}=1$,
{\rm (2)} $1<i(\Phi)\le\infty$ and $\|f\|_{\LPp(\R^n,w)}=1$,
{\rm (3)} $1<i(\Phi)\le\infty$ and $\|f\|_{\wLPp(\R^n,w)}=1$, 
then
\begin{equation}\label{Mf2}
 M(f\chi_{(2B)^{\complement}})(x)\le C_0\Phi^{-1}(\vp(B)),
 \quad x\in B,
\end{equation}
where the constant $C_0$ is independent of $f$ and $B$.
\end{lem}

\begin{proof}
Let $f_2=f\chi_{(2B)^{\complement}}$,
$B=B(a,r)$ and $x\in B$.
We show that, for all balls $B'\ni x$,
\begin{equation*} 
 \frac{1}{|B'|}\int_{B'} |f_2(x)|\,dx
 \ls
 \Phi^{-1}(\vp(B)).
\end{equation*}
Let $B'=B(z,r')$.
If $r'\le r/2$, 
then $\int_{B'} |f_2(y)|\,dy=0$, since $B'\subset2B$.
If $r'> r/2$, 
then $B'\subset B(a,3r')$.
Setting $B''=B(a,3r')$, we have 
\begin{equation*}
 \frac{1}{|B'|}\int_{B'} |f_2(x)|\,dx
 \ls
 \frac{1}{|B''|}\int_{B''} |f_2(x)|\,dx.
\end{equation*}
If we show 
\begin{equation}\label{B''}
 \frac{1}{|B''|}\int_{B''} |f_2(x)|\,dx
 \ls
 \Phi^{-1}(\vp(B'')),
\end{equation}
then we have \eqref{Mf2}, 
since $\vp$ is almost decreasing and $\Phi^{-1}$ satisfies the doubling condition.

\noindent
Case (1): We use \eqref{fint B} and \eqref{fint 1}.
Since $w\in A_1$, we have
\begin{equation*}
 \frac{1}{|B''|}\int_{B''} |f_2(x)|\,dx
 \ls
 \frac{1}{w(B'')}\int_{B''} |f_2(x)|w(x)\,dx
 \ls
 \Phi^{-1}(\vp(B'')).
\end{equation*}

\noindent
Case (2): We use \eqref{fint B} and \eqref{fint p}.
Since $i(\Phi)>1$ and $w\in A_{i(\Phi)}$,
we can take $p\in(1,i(\Phi))$ such that $w\in A_p$.
In this case 
$t\mapsto\Phi(t)/t^p$ is almost increasing and
\begin{equation*}
 \frac{1}{|B''|}\int_{B''} |f_2(x)|\,dx
 \ls
 \left(\frac{1}{w(B'')}\int_{B''} |f_2(x)|^pw(x)\,dx\right)^{1/p}
 \ls
 \Phi^{-1}(\vp(B'')).
\end{equation*}

\noindent
Case (3): We use \eqref{fint B} and \eqref{fint q}.
Since $i(\Phi)>1$ and $w\in A_{i(\Phi)}$,
we can take $q\in(1,i(\Phi))$ such that $w\in A_q$.
In this case 
$t\mapsto\Phi(t)/t^p$ is almost increasing for $p\in(q,i(\Phi))$ and
\begin{equation*}
 \frac{1}{|B''|}\int_{B''} |f_2(x)|\,dx
 \ls
 \left(\frac{1}{w(B'')}\int_{B''} |f_2(x)|^qw(x)\,dx\right)^{1/q}
 \ls
 \Phi^{-1}(\vp(B'')).
\end{equation*}
Therefore, we have the conclusion.
\end{proof}

\begin{lem}\label{lem:Tf2}
Let $\Phi\in\bcY$, $w\in A_{i(\Phi)}$ and $\vp\in\cGdec_w$.
Assume that $\vp$ satisfies \eqref{int vp x}.
Let $B$ be a ball.
If one of the following three conditions holds;
{\rm (1)} $i(\Phi)=1$ and $\|f\|_{\LPp(\R^n,w)}=1$,
{\rm (2)} $1<i(\Phi)\le I(\Phi)<\infty$ and $\|f\|_{\LPp(\R^n,w)}=1$,
{\rm (3)} $1<i(\Phi)\le I(\Phi)<\infty$ and $\|f\|_{\wLPp(\R^n,w)}=1$, 
then 
\begin{equation*} 
 \int_{\R^n\setminus{2B}}|K(x,y)f(y)|\,dy
 \le
 C_0 \Phi^{-1}(\vp(B)),
 \quad x\in B,
\end{equation*}
where the constant $C_0$ is independent of $f$ and $B$.
\end{lem}

\begin{proof}
By Remark~\ref{rem:index 2}~\ref{item:bdtwo}
we may assume that $\Phi\in\dtwo$.
Let $B=B(a,r)$ and $B_k=B(a,2^kr)$, $k=1,2,\ldots$.
Then
\begin{align*}
 \int_{\R^n\setminus{2B}} |K(x,y)f(y)|\,dy 
 &=
 \sum_{k=2}^{\infty}\int_{B_k\setminus B_{k-1}} |K(x,y) f(y)|\,dy
\\
 &\ls
 \sum_{k=2}^\infty \frac{1}{|B_k|}\int_{B_k} |f(y)|\,dy.
\end{align*}
For each case of (1), (2) and (3),
by the same way as in the proof of the previous lemma,
we have 
\begin{equation*} 
 \frac{1}{|B_k|}\int_{B_k} |f(y)|\,dy
 \ls
 \Phi^{-1}(\vp(B_k)),
\end{equation*}
instead of \eqref{B''}.
By the doubling condition of $\Phi^{-1}(\vp(\cdot))$ and Lemma~\ref{lem:int Phi vp}
we have
\begin{align*}
 \int_{\R^n\setminus{2B}} |K(x,y)f(y)|\,dy 
 &\ls
  \sum_{k=2}^\infty \Phi^{-1}(\vp(B_k))
\\ 
 &\sim
 \sum_{k=2}^\infty
 \int_{2^{k-1}r}^{2^{k}r}\frac{\Phi^{-1}(\vp(a,t))}{t}\,dt
\\
 &\le 
 \int_r^\infty\frac{\Phi^{-1}(\vp(a,t))}{t}\,dt
 \ls 
 \Phi^{-1}(\vp(B)),
\end{align*}
which shows the conclusion.
\end{proof}

Now we prove Theorem~\ref{thm:M CZO}.

\begin{proof}[\bf Proof of Theorem~\ref{thm:M CZO} (i)]\label{proof of Theorem CZO (i)}
Let $f \in \LPp(\R^n,w)$ or $f \in \wLPp(\R^n,w)$.
We may assume that $\|f\|_{\LPp(\R^n,w)}=1$ or $\|f\|_{\wLPp(\R^n,w)}=1$, respectively.
We will show that $\|Mf\|_{\Phi,\vp,w,B}\le C$ or $\|Mf\|_{\Phi,\vp,w,B,\weak}\le C$
for any ball $B=B(a,r)$.
Let $f=f_1+f_2$ with $f_1 = f\chi_{2B}$.
If $i(\Phi)=1$ and $\|f\|_{\LPp(\R^n,w)}=1$, or, 
if $1<i(\Phi)\le\infty$ and $\|f\|_{\wLPp(\R^n,w)}=1$,
then
$\|Mf_1\|_{\Phi,\vp,w,B,\weak}\le C$
by Lemma~\ref{lem:Mf1}.
If $1<i(\Phi)\le\infty$ and $\|f\|_{\LPp(\R^n,w)}=1$,
then
$\|Mf_1\|_{\Phi,\vp,w,B}\le C$
by Lemma~\ref{lem:Mf1}.
Moreover, by Lemma~\ref{lem:Mf2} we have
\begin{equation*}
 \|Mf_2\|_{\Phi,\vp,w,B,\weak}
 \le
 \|Mf_2\|_{\Phi,\vp,w,B}
 \le
 C_0.
\end{equation*}
The proof is complete.
\end{proof}

\begin{proof}[\bf Proof of Theorem~\ref{thm:M CZO} (ii)]\label{proof of Theorem CZO (ii)}
Let $f \in \LPp(\R^n,w)$ or $f \in \wLPp(\R^n,w)$.
We may assume that $\|f\|_{\LPp(\R^n,w)}=1$ or $\|f\|_{\wLPp(\R^n,w)}=1$, respectively.
For any ball $B=B(a,r)$,
let $f=f_1+f_2$ with $f_1 = f\chi_{2B}$,
and 
let
\begin{equation}\label{Tf def}
 Tf(x)=Tf_1(x)+\int_{\R^n}K(x,y)f_2(y)\,dy,
 \quad x\in B.
\end{equation}
We will show that $Tf(x)$ in \eqref{Tf def} is well defined and
independent of the choice of $B$ containing $x$
and that $T$ is bounded.

For the part $Tf_1$, 
by Lemma~\ref{lem:Mf1},
if $i(\Phi)=1$ and $\|f\|_{\LPp(\R^n,w)}=1$, or, 
if $1<i(\Phi)\le I(\Phi)<\infty$ and $\|f\|_{\wLPp(\R^n,w)}=1$,
then
\begin{equation}\label{Tf1 w}
 \|Tf_1\|_{\Phi,\vp,w,B,\weak}\le C.
\end{equation}
If $1<i(\Phi)\le I(\Phi)<\infty$ and $\|f\|_{\LPp(\R^n,w)}=1$,
then
\begin{equation}\label{Tf1}
 \|Tf_1\|_{\Phi,\vp,w,B}\le C.
\end{equation}
Moreover, 
by Lemma~\ref{lem:Tf2}
we have
\begin{equation*}
 \left|\int_{\R^n} K(x,y)f_2(y)\,dy\right|
 \le
 \int_{\R^n\setminus{2B}}|K(x,y)f(y)|\,dy
 \le
 C_0 \Phi^{-1}(\vp(B)),
 \quad x\in B.
\end{equation*}
Then
\begin{align*}
 \int_B
  \Phi\left(
    \frac{\left|\int_{\R^n} K(x,y)f_2(y)\,dy\right|}{C_0}
  \right)
 w(x)\,dx
 &\le
 \int_B
  \Phi\left(
   \Phi^{-1}(\vp(B))
  \right)
 w(x)\,dx
\\
 &=
 \vp(B)w(B),
\end{align*}
that is,
\begin{equation}\label{Tf2}
 \left\|\int_{\R^n}K(\cdot,y)f_2(y)\,dy\right\|_{\Phi,\vp,w,B,\weak}
 \le
 \left\|\int_{\R^n}K(\cdot,y)f_2(y)\,dy\right\|_{\Phi,\vp,w,B}
 \le
 C_0.
\end{equation}

Moreover, if $x \in B\cap B'$ and
\begin{equation*}
 f = f_1 + f_2 = g_1 + g_2, 
 \quad f_1 = f\chi_{2B},
 \quad g_1 = g\chi_{2B'}
\end{equation*}
then 
$\supp(f_2 - g_2)$ is compact and $x \notin \supp(f_2 - g_2)$. 
From \eqref{CZ3}, it follows that
\begin{equation*}
 \int_{\R^n} K(x, y)\left(f_2(y) - g_2(y) \right)\,dy
 =
 T(f_2-g_2)(x).
\end{equation*}
Hence
\begin{equation*}
 \left(Tf_1(x) + \int_{\R^n} K(x,y)f_2(y) dy\right)
 -
 \left(Tg_1(x) + \int_{\R^n} K(x,y)g_2(y) dy\right)
 = 0.
\end{equation*}
Therefore, $Tf(x)$ in \eqref{Tf def} is well defined and
independent of the choice of $B$ containing $x$.
Further, by \eqref{Tf1 w}, \eqref{Tf1} and \eqref{Tf2} we have
\begin{equation*}
 \|Tf\|_{\Phi,\vp,w,B,\weak}\le C 
 \quad\text{or}\quad
 \|Tf\|_{\Phi,\vp,w,B}\le C,
 \quad\text{for all balls $B$},
\end{equation*}
which shows the conclusion.
\end{proof}

\section*{Acknowledgement}
The authors would like to thank the referee for her/him careful reading 
and useful comments.
The second author was supported by Grant-in-Aid for Scientific Research (B), 
No.~15H03621 and 20H01815, Japan Society for the Promotion of Science.



\begin{thebibliography}{99}

\bibitem{Coifman1972}
R.~R.~Coifman, 
Distribution function inequalities for singular integrals, 
Proc. Nat. Acad. Sci. U.S.A. 69 (1972), 2838--2839. 

\bibitem{Coifman-Fefferman1974}
R.~R.~Coifman and C.~Fefferman, 
Weighted norm inequalities for maximal functions and singular integrals, 
Studia Math. 51 (1974), 241--250.

\bibitem{Curbera-GarciaCuerva-Martell-Perez2006}
G.~P.~Curbera, J.~Garc\'ia-Cuerva, J.~M.~Martell and C.~P\'erez, 
Extrapolation with weights, rearrangement-invariant function spaces, 
modular inequalities and applications to singular integrals,
Adv. Math. 203 (2006), No.~1, 256--318. 

\bibitem{Deringoz-Guliyev-Nakai-Sawano-Shi2019Posi} 
F.~Deringoz, V.S.~Guliyev, E.~Nakai, Y.~Sawano and M.~Shi,
Generalized fractional maximal and integral operators 
on Orlicz and generalized Orlicz-Morrey spaces of the third kind,
Positivity 23 (2019), No.~3, 727--757.

\bibitem{Deringoz-Guliyev-Samko2014} 
F.~Deringoz, V.S.~Guliyev and S.~Samko,
Boundedness of the maximal and singular operators on generalized Orlicz-Morrey spaces. 
Operator theory, operator algebras and applications, 139--158, Oper. Theory Adv. Appl., 242, 
Birkh\"auser/Springer, Basel, 2014. 

\bibitem{Fu-Yang-Yuan2012}
X.~Fu, D.~Yang and W.~Yuan, 
Boundedness of multilinear commutators of Calder\'on-Zygmund operators 
on Orlicz spaces over non-homogeneous spaces. 
Taiwanese J. Math. 16 (2012), No.~6, 2203--2238.

\bibitem{Gala-Sawano-Tanaka2015}  
S.~Gala, Y.~Sawano and H.~Tanaka, 
A remark on two generalized Orlicz-Morrey spaces, 
J. Approx. Theory 198 (2015), 1--9. 

\bibitem{GarciaCuerva-RubiodeFrancia1985}
J.~Garc\'ia-Cuerva and J.~Rubio de Francia, 
Weighted norm inequalities and related topics. 
North-Holland Mathematics Studies, 
116. Notas de Matem\'atica [Mathematical Notes], 104. 
North-Holland Publishing Co., Amsterdam, 1985. 
x+604 pp. ISBN: 0-444-87804-1 

\bibitem{Grafakos2014GTM249}
L.~Grafakos, 
Classical Fourier analysis. Third edition. 
Graduate Texts in Mathematics, 249. 
Springer, New York, 2014. 

\bibitem{Guliyev-Hasanov-Sawano-Noi2016}
V.~S.~Guliyev, S.~G.~Hasanov, Y.~Sawano and T.~Noi, 
Non-smooth atomic decompositions for generalized Orlicz-Morrey spaces of the third kind,
Acta Appl. Math. 145 (2016), No.~1, 133--174. 

\bibitem{Gustavsson-Peetre1977}
J.~Gustavsson and J.~Peetre,  
Interpolation of Orlicz spaces, 
Studia Math. 60 (1977),  No.~1, 33--59. 

\bibitem{Ho2013} 
K.-P.~Ho, 
Vector-valued maximal inequalities on weighted Orlicz-Morrey spaces, 
Tokyo J. Math. 36 (2013), No.~2, 499--512.

\bibitem{Ho2019} 
K.-P.~Ho, 
Weak type estimates of singular integral operators on Morrey-Banach spaces, 
Integral Equations Operator Theory 91 (2019), No.~3, Paper No.~20, 18 pp. 

\bibitem{Ho2020} 
K.-P.~Ho, 
Definability of singular integral operators on Morrey-Banach spaces, 
J. Math. Soc. Japan 72 (2020), No.~1, 155--170. 

\bibitem{Hunt-Muckenhoupt-Wheeden1973}
R.~Hunt, B.~Muckenhoupt and R.~Wheeden, 
Weighted norm inequalities for the conjugate function and Hilbert transform, 
Trans. Amer. Math. Soc. 176 (1973), 227--251. 

\bibitem{Kawasumi-Nakai2020Hiroshima}
R.~Kawasumi and E.~Nakai,
Pointwise multipliers on weak Orlicz spaces,
Hiroshima Math. J. 50 (2020), No.~2, 169--184.

\bibitem{Kawasumi-Nakai-Shi2021MathNachr}
R.~Kawasumi, E.~Nakai and M.~Shi,  
Characterization of the boundedness of generalized fractional integral
and maximal operators on Orlicz-Morrey and weak Orlicz-Morrey spaces,
to appear in Math. Nachr. \\
https://arxiv.org/abs/2107.10553

\bibitem{Kita2009}
H.~Kita,
Orlicz spaces and their applications (Japanese),
Iwanami Shoten, Publishers. Tokyo, 2009.

\bibitem{Kokilashvili-Krbec1991}
V.~Kokilashvili and M.~Krbec,
{Weighted inequalities in Lorentz and Orlicz spaces},
World Scientific Publishing Co., Inc., River Edge, NJ, 1991. 

\bibitem{Komori2012}
Y.~Komori-Furuya,
Weak $L^p$-boundedness of the Hardy-Littlewood maximal function (in Japanese),
Report collection of Harmonic Analysis Seminar 2011 at Osaka University, 39--42, printed in 2012.

\bibitem{Komori-Shirai2009}
Y.~Komori and S.~Shirai, 
Weighted Morrey spaces and a singular integral operator. 
Math. Nachr. 282 (2009), No.~2, 219--231. 

\bibitem{Liu-Wang2013}
P.~Liu and M.~Wang,
Weak Orlicz spaces: Some basic properties and their applications to harmonic analysis,
Sci. China. Math. 56 (2013), 789--802.

\bibitem{Krasnoselsky-Rutitsky1961}
M,~A.~Krasnoselsky and Y.~B.~Rutitsky, 
Convex functions and Orlicz spaces. 
Translated from the first Russian edition by Leo F. Boron. 
P. Noordhoff Ltd., Groningen 1961.

\bibitem{Maligranda1985}
L.~Maligranda, 
Indices and interpolation, 
Dissertationes Math. (Rozprawy Mat.)  234 (1985), 49 pp. 

\bibitem{Maligranda1989}
L. Maligranda,
Orlicz spaces and interpolation,
Seminars in mathematics 5,
Departamento de Matem\'atica, Universidade Estadual de Campinas, Brasil, 1989.

\bibitem{Muckenhoupt1972}
B.~Muckenhoupt, 
Weighted norm inequalities for the Hardy maximal function, 
Trans. Amer. Math. Soc. 165 (1972), 207--226.






\bibitem{Nakai2004KIT}
E. Nakai,
{Generalized fractional integrals on Orlicz-Morrey spaces}, 
Banach and Function Spaces
(Kitakyushu, 2003), Yokohama Publishers, Yokohama, 2004, 323--333.


\bibitem{Nakai2008Studia}
E.~Nakai,
{Orlicz-Morrey spaces and the Hardy-Littlewood maximal function},
Studia Math. 188 (2008), No~3, 193--221.

\bibitem{Nakai2008KIT}
E.~Nakai, 
Calder\'on-Zygmund operators on Orlicz-Morrey spaces and modular inequalities.  
Banach and function spaces II, 393--410, Yokohama Publ., Yokohama, 2008. 





\bibitem{ONeil1965}
R.~O'Neil,
{Fractional integration in Orlicz spaces. I.}, 
Trans. Amer. Math. Soc. 115 (1965), 300--328. 

\bibitem{Sawano2019} 
Y.~Sawano, 
Singular integral operators acting on Orlicz-Morrey spaces of the first kind,
Nonlinear Stud.  26  (2019),  No.~4, 895--910. 

\bibitem{Sawano-Sugano-Tanaka2012}
Y.~Sawano, S.~Sugano and H.~Tanaka,  
Orlicz-Morrey spaces and fractional operators,
Potential Anal. 36 (2012), No.~4, 517--556. 

\bibitem{Shi-Arai-Nakai2019Taiwan}
M.~Shi, R.~Arai and E.~Nakai, 
Generalized fractional integral operators 
and their commutators with functions in generalized Campanato spaces on Orlicz spaces,
Taiwanese J. Math. 23 (2019), No. 6, 1339--1364.

\bibitem{Shi-Arai-Nakai2021Banach}
M.~Shi, R.~Arai and E.~Nakai, 
Commutators of integral operators with functions in Campanato spaces on Orlicz-Morrey spaces,
Banach J. Math. Anal. 15 (2021), No.~1, Paper No.~22, 41pp. 

\bibitem{Torchinsky1976}
A. Torchinsky, 
Interpolation of operations and Orlicz classes, 
Studia Math. 59 (1976), No.~2, 177--207. 

\bibitem{Yabuta1985}
K.~Yabuta, 
Generalizations of Calder\'on-Zygmund operators,
Studia Math. 82 (1985), 17--31.

\end{thebibliography}
\end{document}